\documentclass{article}
\usepackage[utf8]{inputenc}

\title{Sidon sets and statistics of the ElGamal function}
\author{Lucas Boppr\'e Niehues, Joachim von zur Gathen,\\Lucas Pandolfo Perin, Ana Zumalac\'arregui}
\date{\today}

\usepackage{graphicx}
\usepackage{amsmath}
\usepackage{amsthm}
\usepackage{amsfonts}
\usepackage{caption}
\usepackage{subcaption}

\newtheorem{thm}{Theorem}

\newtheorem{lem}{Lemma}

\numberwithin{equation}{section} \numberwithin{thm}{section}
\numberwithin{lem}{section}
\numberwithin{problem}{section}
\numberwithin{prop}{section}
\numberwithin{cor}{section}
\numberwithin{conj}{section}
\numberwithin{obs}{section}

\newcommand{\Z}{{\mathbb Z}}
\newcommand{\F}{{\mathbb F}}
\newcommand{\twodots}{\ .. \ }

\begin{document}

\maketitle

\begin{abstract}
In the ElGamal signature and encryption schemes, an element $x$ of the underlying group
$G = \Z_p^\times = \{1, \ldots, p-1 \}$ for a prime $p$
is also considered as an exponent, for example in $g^x$, where $g$ is a generator of G.
This \emph{ElGamal map} $x \mapsto g^x$ is poorly understood,
and one may wonder whether it has some randomness properties.
The underlying map from $G$ to $\Z_{p-1}$ with $x \mapsto x$ is trivial from a
computer science point of view, but does not seem to have any mathematical structure.

This work presents two pieces of evidence for randomness.
Firstly, experiments with small primes suggest that the map behaves like a uniformly
random permutation with respect to two properties that we consider.
Secondly, the theory of Sidon sets shows that the graph of this map is
equidistributed in a suitable sense.

It remains an open question to prove more randomness properties, 
for example, that the ElGamal map is pseudorandom.
\end{abstract}

\section{Introduction}
In the ElGamal signature scheme \cite{elgamal} with parameter $n$,
 we take an $n$-bit number $d$ and a cyclic group $G = \langle g \rangle$  of order $d$.
In ElGamal's original proposal,  $p$ is an $n$-bit prime number, 
$G = \mathbb{Z}_p^\times = \{1, \ldots, p-1\}$, $d = p-1$, and 
$\mathbb{Z}_d = \{1, \ldots, d\}$ is the \emph{exponent group}.
More commonly, one takes $\mathbb{Z}_d = \{0, \ldots, d-1\}$, 
but both are valid set of representatives.
We let $g$ be a generator of $G$, so that $G = \{g^b \colon b \in \Z_d \}$.
The object of this paper is to investigate randomness properties of the
\emph{ElGamal map} from $G$ to $G$ with $x \mapsto g^x$,
where $x \in \Z_d$ on the right hand side.
Since $g^x$ determines $x$ uniquely, this is a permutation of $G$.
If we consider $x \in \Z_d$ on the left hand side, it is the discrete exponentiation
map in base $g$.

A secret global key $a \in \Z_d$ and session key $k \in \Z_d^\times$ are chosen
uniformly at random, and their public versions $A = g^a$ and $K = g^k$ in $G$
are published. The signature of a message $m \in \Z_d$ is $(K, b)$ with
$b = k^{-1} (m-aK) \in \Z_d$.

The private key is easily broken if discrete
logarithms in $G$ can be calculated efficiently; see Figure \ref{fig:dexpdlog}.
For more details, see von zur Gathen \cite{cryptoschool}, Sections 8.2 and 9.8.

\begin{figure}[ht!]
\centering
\includegraphics[width=\textwidth]{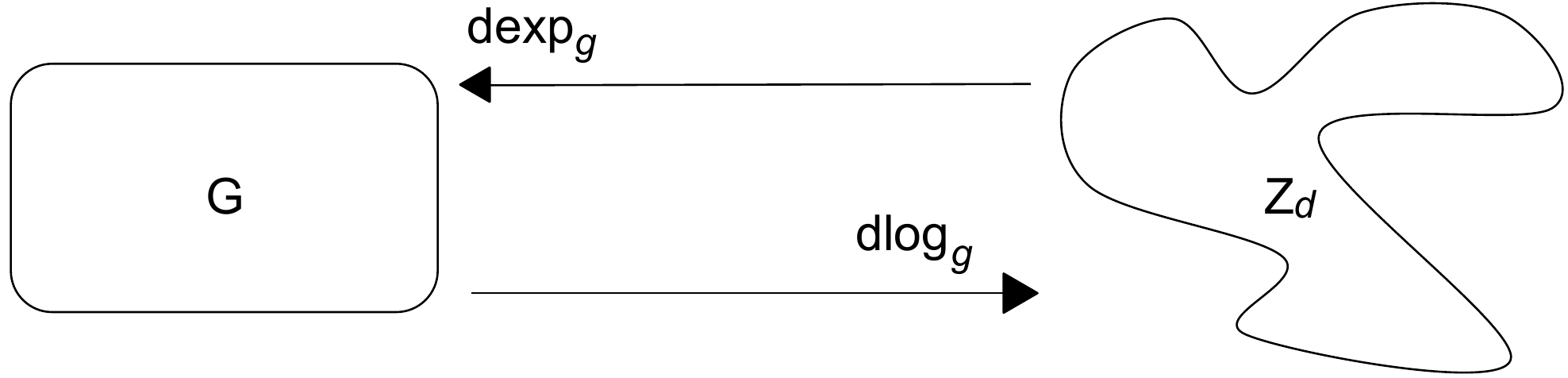}
\caption{Public group and exponent group in ElGamal Signature scheme}
\label{fig:dexpdlog}
\end{figure}

The Decisional Diffie-Hellman (DDH) problem is to decide whether, given a triple
$(x,y,z) \in G^3$, there exist $a, b, c \in \Z_d$ so that $x=g^a$, $y=g^b$, and $z = g^{ab}$;
then $(x,y,z)$ is a \emph{Diffie-Hellman triple}.

If such triples are indistinguishable from uniformly random triples, for uniformly random $a$
and $b$, then the ElGamal encryption scheme is indistinguishable by public key
only attacks.
The results of Canetti, Friedlander, Konyagin, Larsen, Lieman, and Shparlinski
\cite{canettiDDH}, indicate that the most significant and least 
significant bits of each element in DDH triples are indeed distributed uniformly.
Do pairs $(x, g^x)$, for uniformly random $x$, exhibit a similar behavior?

This paper first gives some experimental evidence in favor of this.
We take some small primes, just above 1000, and consider two parameters
of permutations: the number of cycles and the number of $k$-cycles for given $k$.
Their averages for random permutations are well-known, and we find that the
average values for the ElGamal function are reasonably close to those numbers.
Secondly, we use the theory of Sidon sets to prove an equidistributional property
with appropriate parameters; see also Cobeli, V\^aj\^aitu \&\ Zaharescu
\cite{CobVajZah} for a different approach to show equidistribution.

Martins \&\ Panario \cite{marpan16} study similar questions, but for general polynomials that need not
be permutations, and for different parameters.
Konyagin, Luca,  Mans, Mathieson, Sha \&\ Shparlinski \cite{konluc16} consider 
enumerative and algorithmic questions about (non-)isomorphic functional graphs,
 and Mans, Sha, Shparlinski \&\ Sutantyo \cite{mansha17}
provide statistics, conjectures, and results about cycle lengths of quadratic polynomials
over finite prime fields. Kurlberg, Luca \&\ Shparlinski
\cite{kurluc15}
and  Felix \&\ Kurlberg \cite{felkur16} deal with fixed points of the map $x \mapsto x^x$ modulo primes.


\section{Experiments in $\mathbb{F}_{p}$}

The pictorial representation in Figure \ref{fig:smiley}
shows the cycle structure of the permutation $x \mapsto g^x$ in $\mathbb{F}_p$ with $p=1009$ and
$g =11$, the smallest generator. Each circle corresponds
to a cycle, whose length is proportional to the circle's circumference.
Next, Figure \ref{fig:smileys} shows together 12 permutations
in $\mathbb{F}_p$ using the 12 smallest generators of $\F_p$.
\begin{figure}[ht!]
\centering
\includegraphics[scale=0.3]{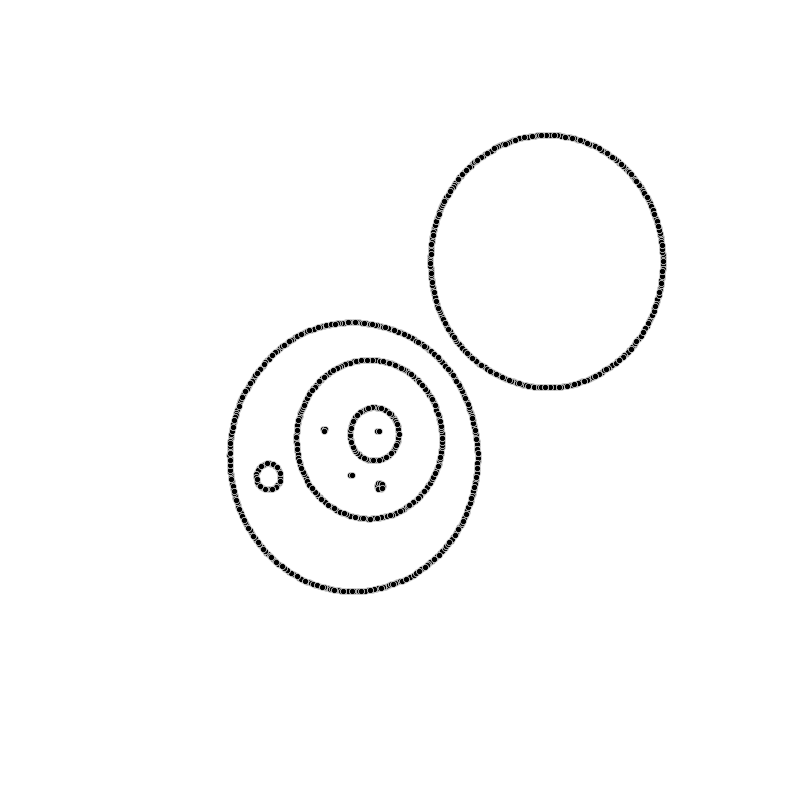}
\caption{Representation of the cycles generated with $g=11$ in $\mathbb{F}_{1009}$}
\label{fig:smiley}
\end{figure}

In the following subsections, we take the cycle structures for all $\phi(1008) = 288$ generators of $\F_{1009}$,
and then of all generators for the first fifty primes larger than 1000.
We calculate the averages for the number of cycles and the number of $k$-cycles and
compare them to the known values for random permutations.

\begin{figure}[ht!]
\centering

\begin{subfigure}{.245\textwidth}
  \centering
  \includegraphics[width=\linewidth]{graph_0011_1009}
\end{subfigure}%
\begin{subfigure}{.245\textwidth}
  \centering
\includegraphics[width=\linewidth]{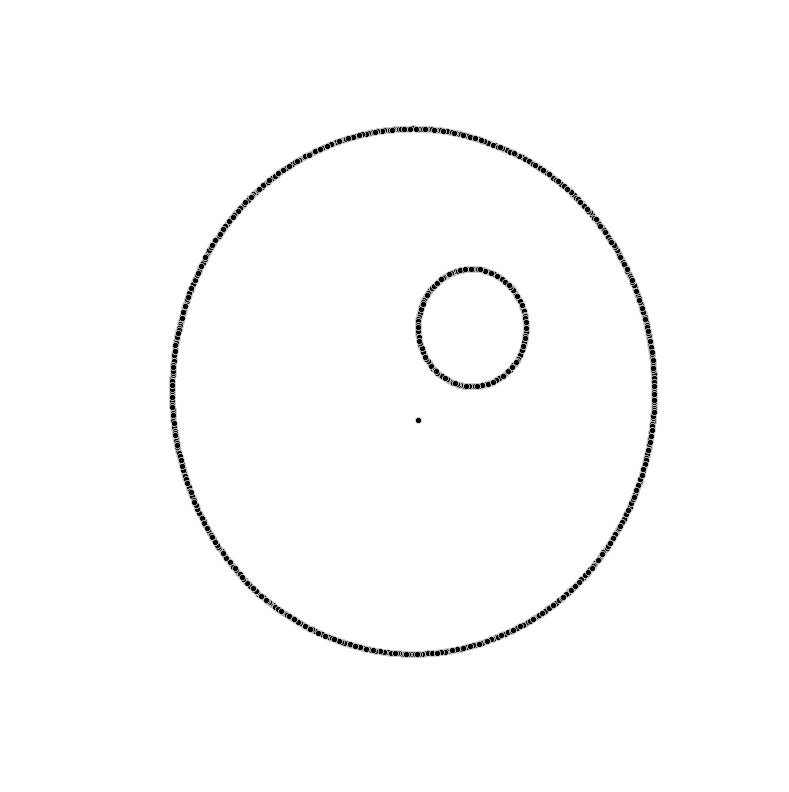}
\end{subfigure}
\begin{subfigure}{.245\textwidth}
  \centering
  \includegraphics[width=\linewidth]{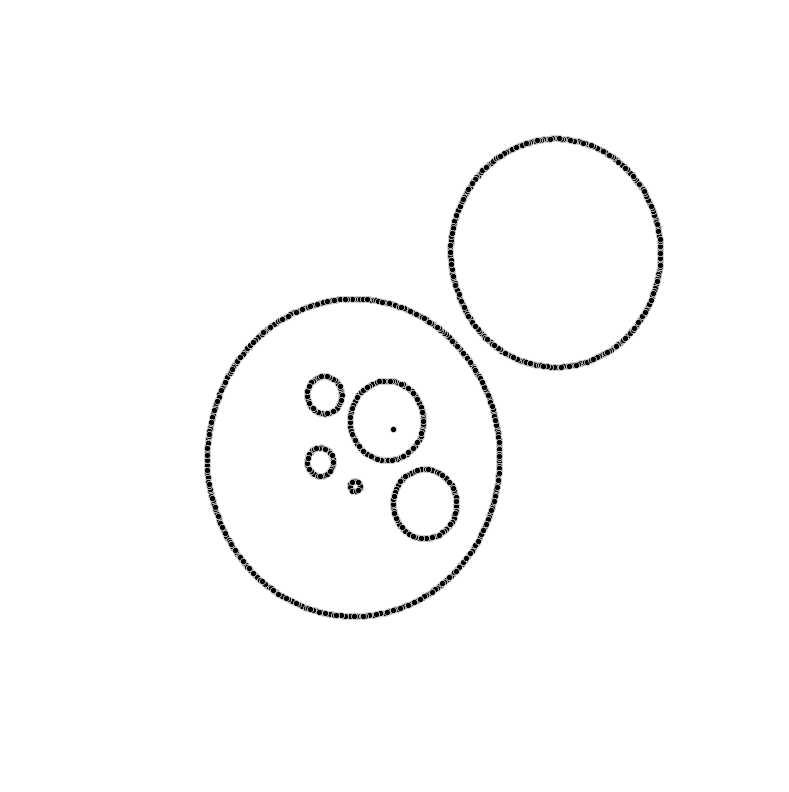}
\end{subfigure}%
\begin{subfigure}{.245\textwidth}
  \centering
  \includegraphics[width=\linewidth]{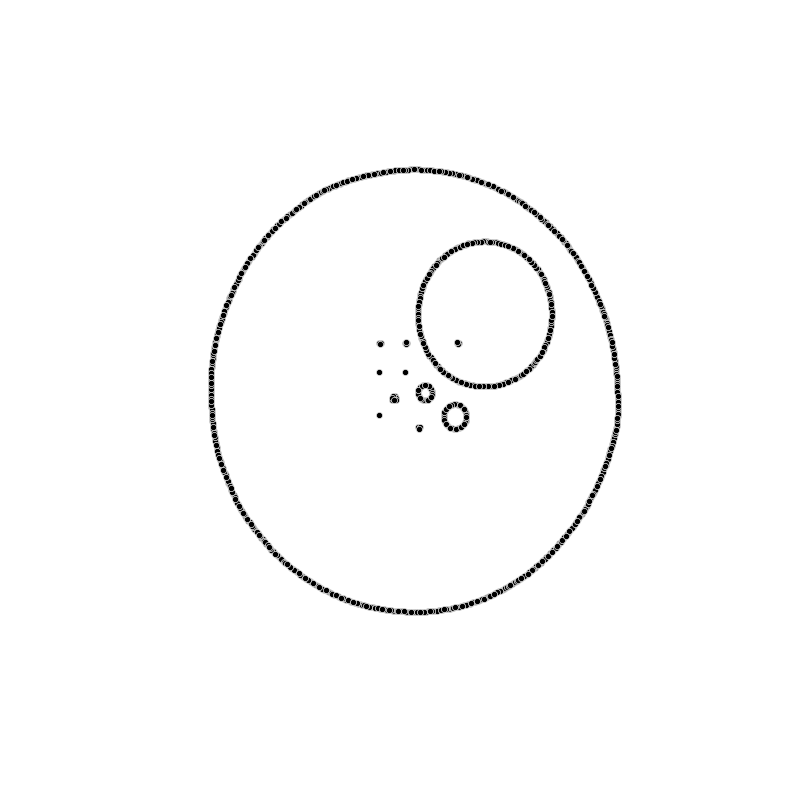}
\end{subfigure}%

\begin{subfigure}{.245\textwidth}
  \centering
\includegraphics[width=\linewidth]{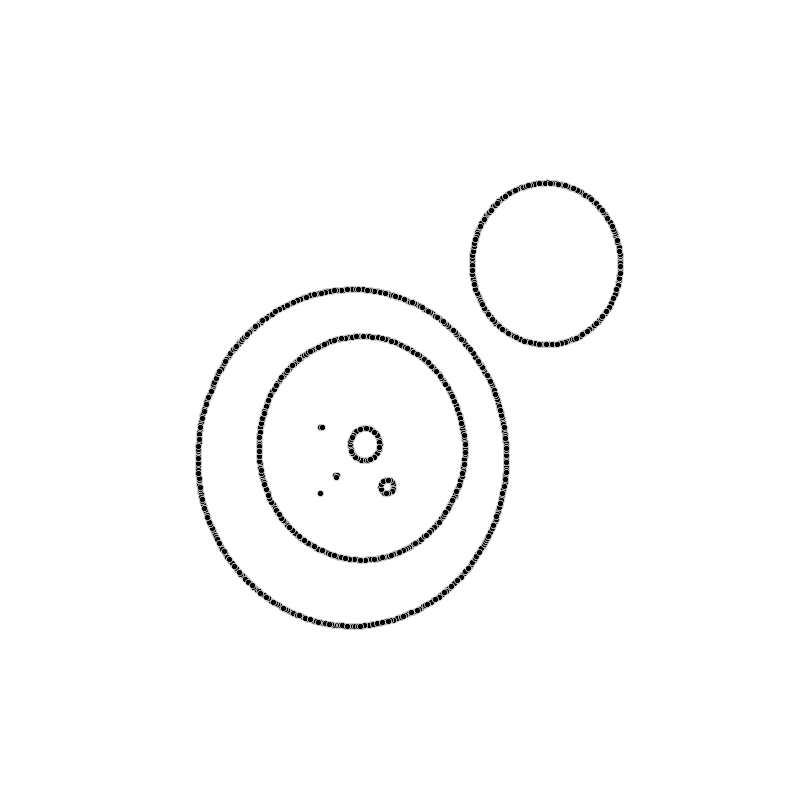}
\end{subfigure}
\begin{subfigure}{.245\textwidth}
  \centering
  \includegraphics[width=\linewidth]{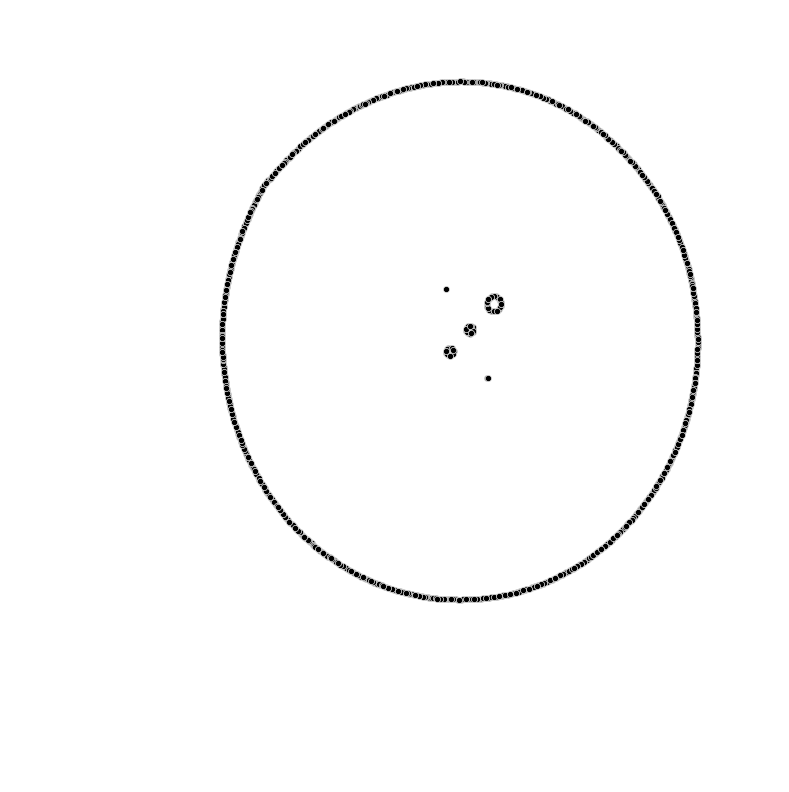}
\end{subfigure}%
\begin{subfigure}{.245\textwidth}
  \centering
  \includegraphics[width=\linewidth]{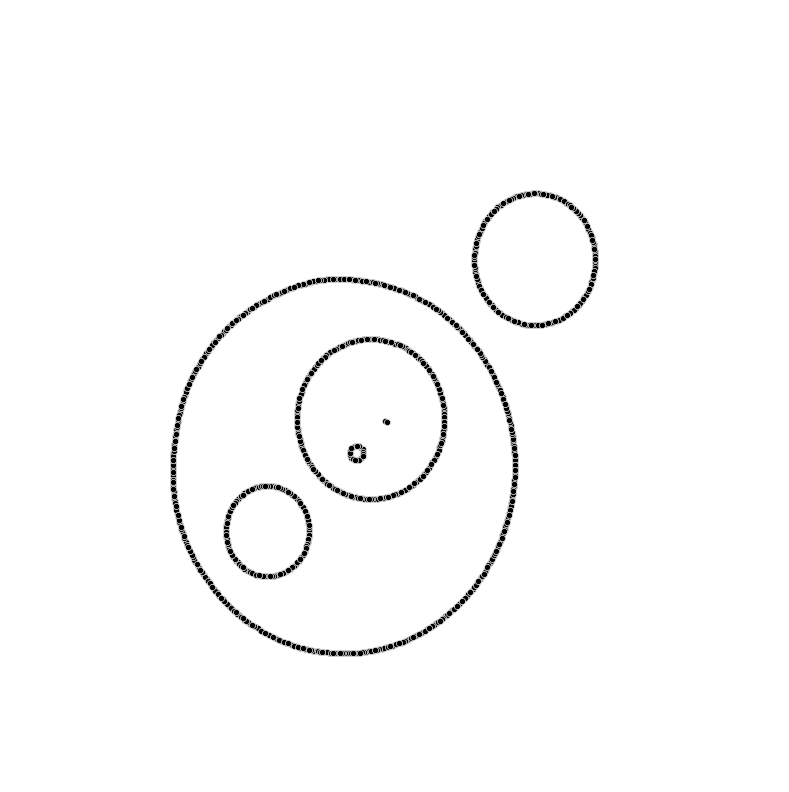}
\end{subfigure}%
\begin{subfigure}{.245\textwidth}
  \centering
\includegraphics[width=\linewidth]{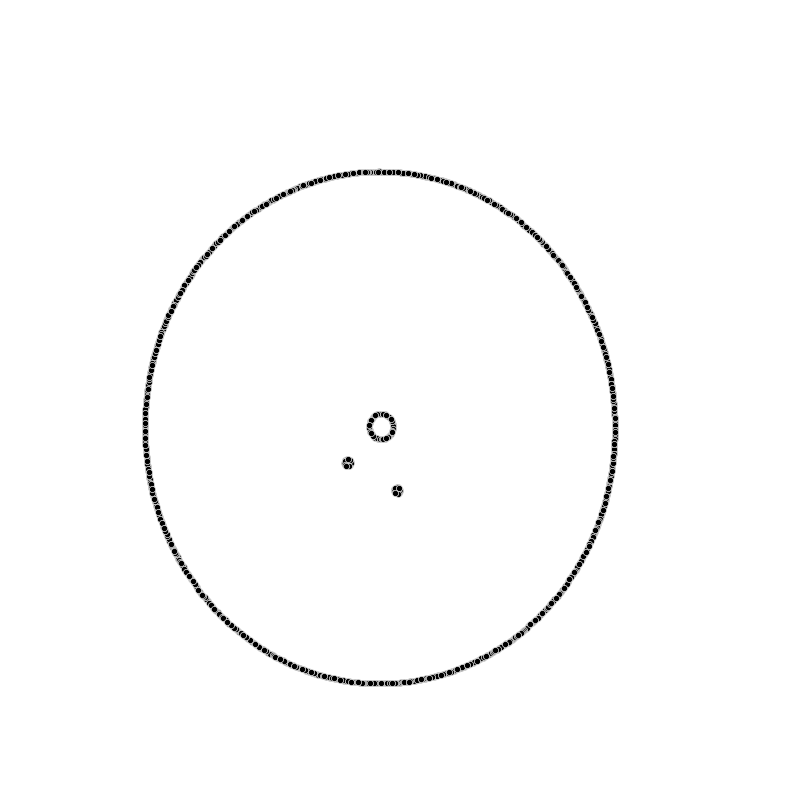}
\end{subfigure}

\begin{subfigure}{.245\textwidth}
  \centering
  \includegraphics[width=\linewidth]{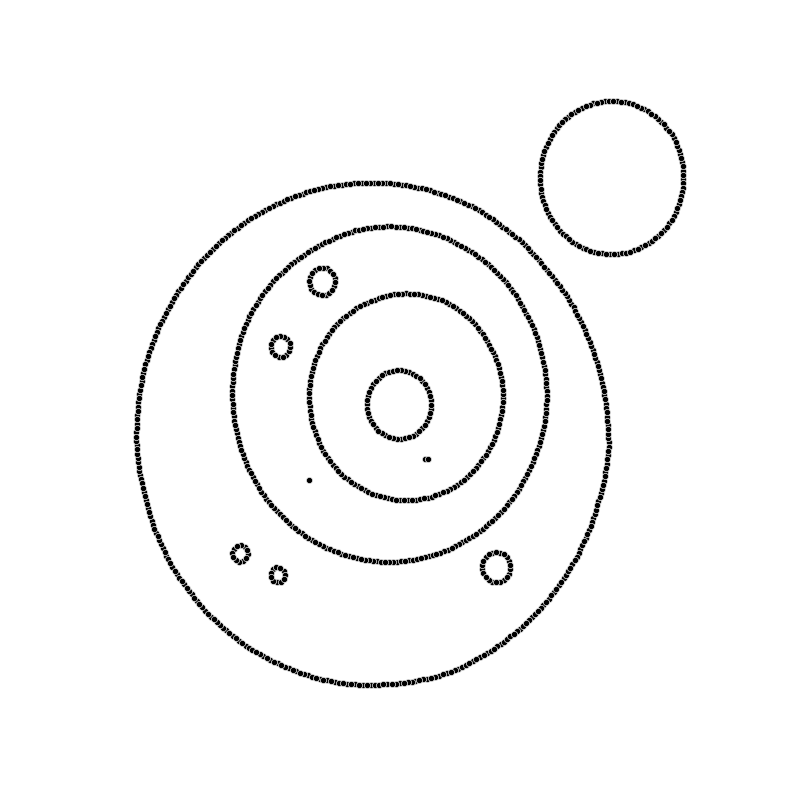}
\end{subfigure}%
\begin{subfigure}{.245\textwidth}
  \centering
  \includegraphics[width=\linewidth]{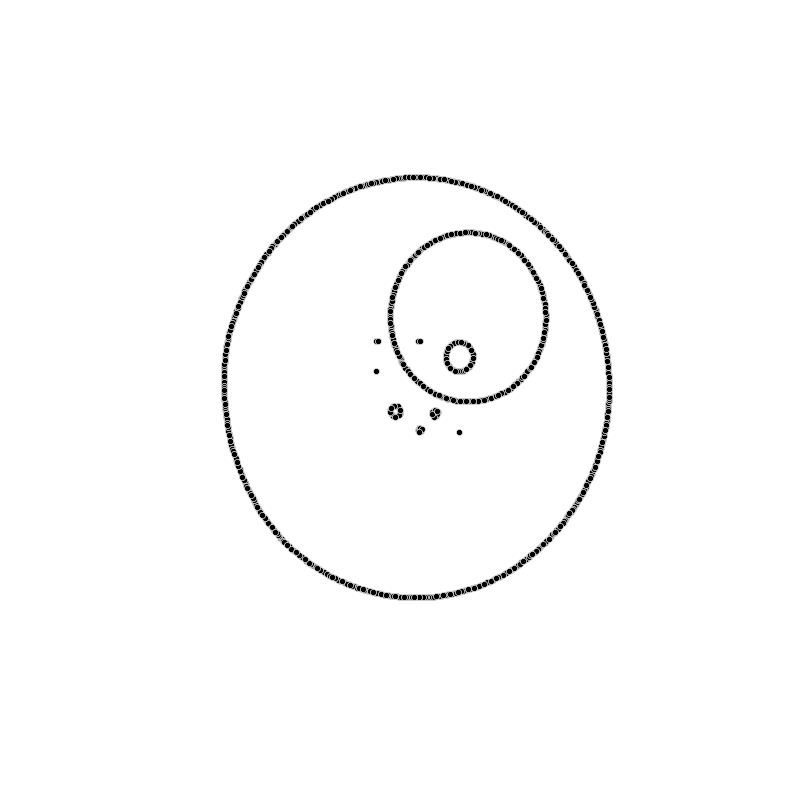}
\end{subfigure}%
\begin{subfigure}{.245\textwidth}
  \centering
\includegraphics[width=\linewidth]{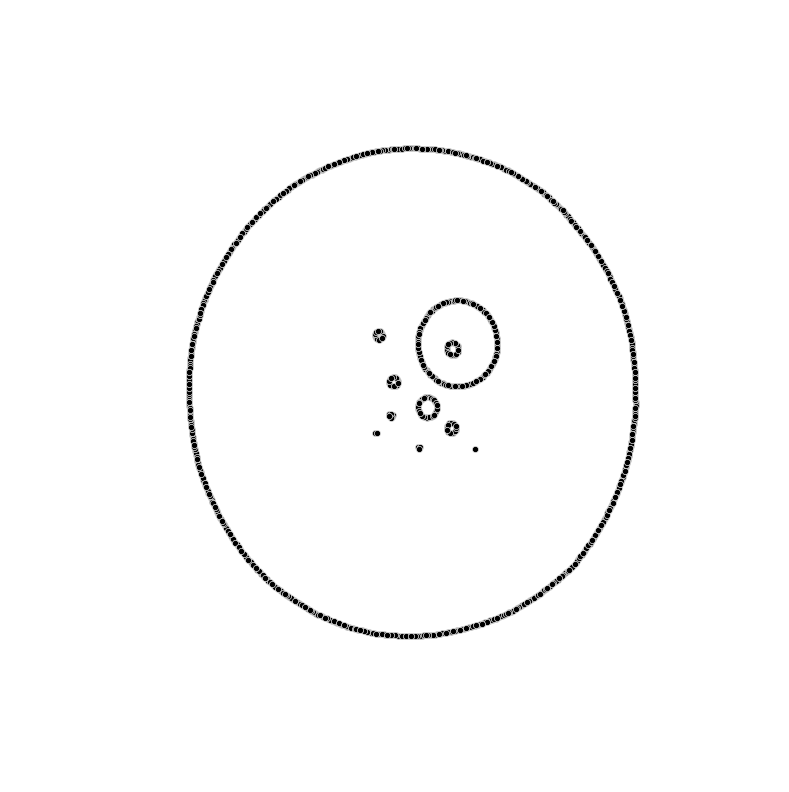}
\end{subfigure}
\begin{subfigure}{.245\textwidth}
  \centering
  \includegraphics[width=\linewidth]{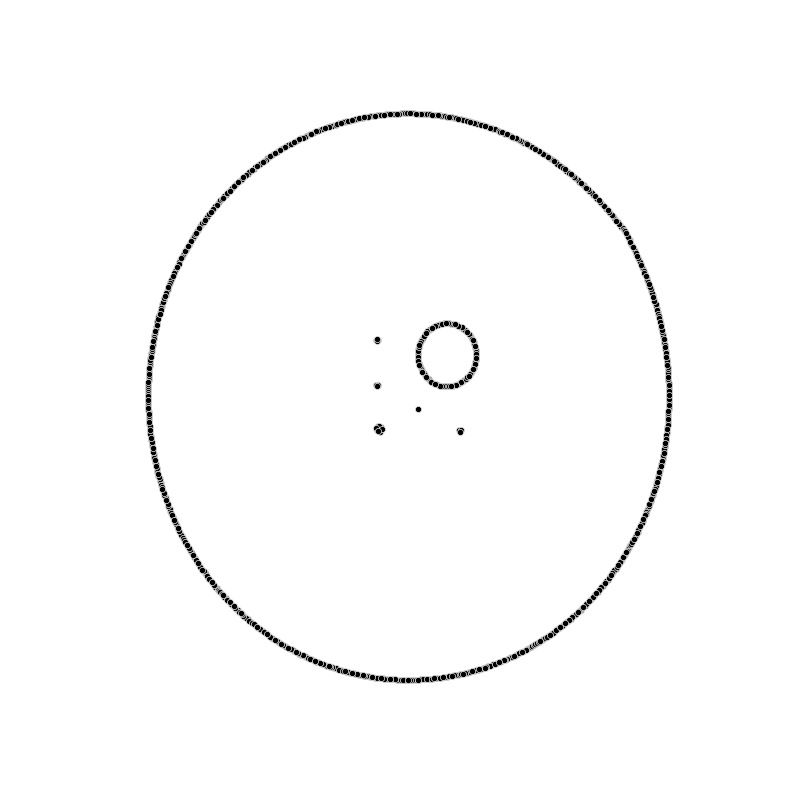}
\end{subfigure}%
 
\caption{Graphical presentation of permutations of $g^{x}$ in $\mathbb{F}_{1009}$}
\label{fig:smileys}
\end{figure}

\subsection{Number of cycles in permutations}

We study in detail the number of cycles in the permutations. The number of
permutations in $S_n$ with $c$ cycles equals the Stirling number $s(n,c)$ of
the first kind, and thus is the coefficient of $x^c$ in the falling
factorial $x^{\underline n} = x \cdot (x-1) \cdots (x-n+1)$. Figure \ref{fig:randomPermCycles}
shows the distribution of the number of cycles for  uniformly random
permutations of $n$ elements, that is, the fraction $s(n,c)/n!$ (in percent) for
$n=1009$ and  $1 \leq c\leq 20$, as a continuous line. In the same figure, the 
experimental statistics for 288 permutations chosen uniformly at random are 
presented as dots. This was done in order to calibrate our expectations.
 Theory and experiments match quite well.
 
Figure \ref{fig:elgamalPermCycles} shows the same continuous line,
but now the dots represent the counts for the 288 generators of $\F_{1009}$.
The result looks quite similar to  Figure~\ref{fig:randomPermCycles}.

\begin{figure}[ht!]
\centering
\includegraphics[width=\textwidth]{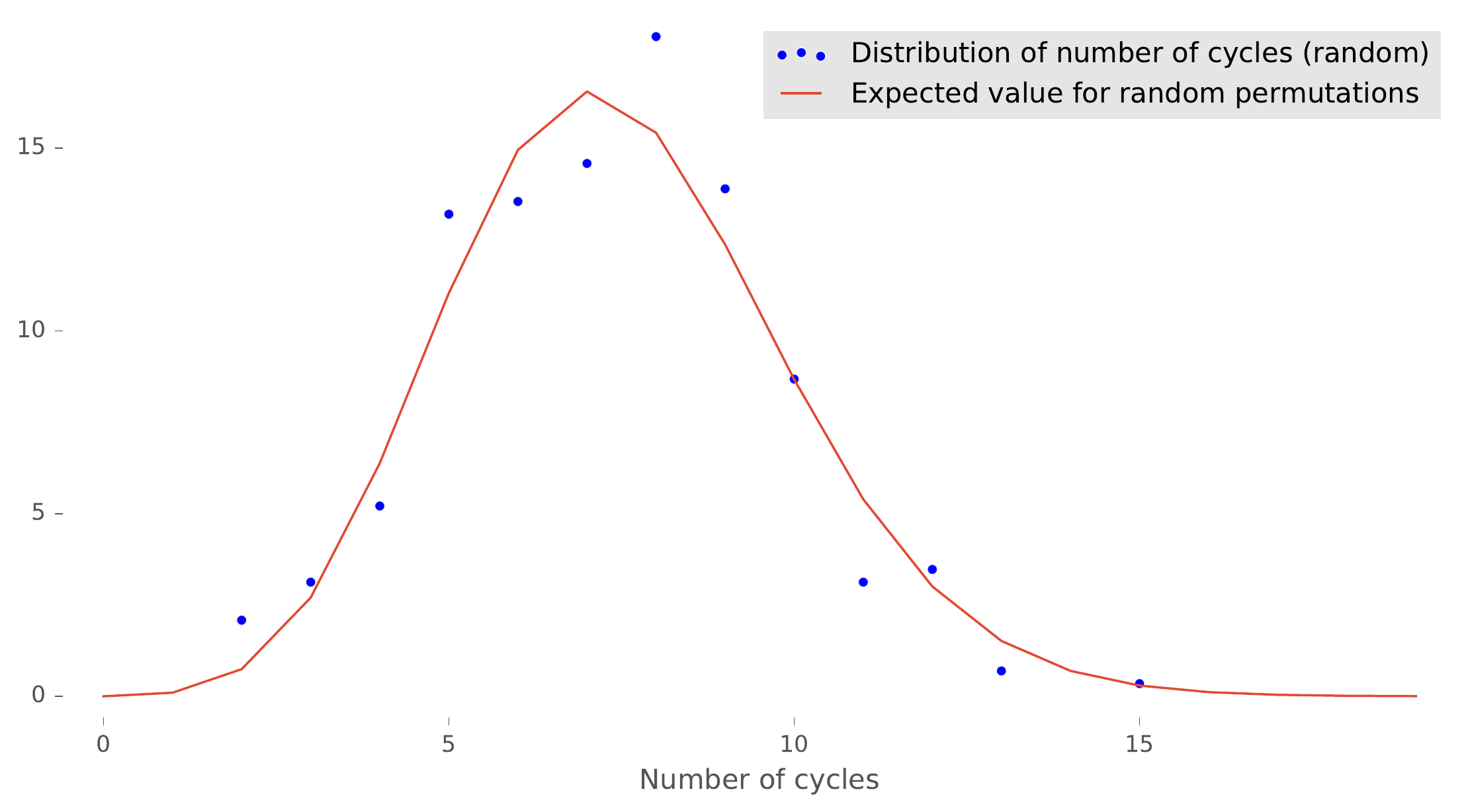}
\caption{Distribution in percent of number of cycles for 288 random permutations in $S_{1009}$.}
\label{fig:randomPermCycles}
\end{figure}

\begin{figure}[ht!]
\centering
\includegraphics[width=\textwidth]{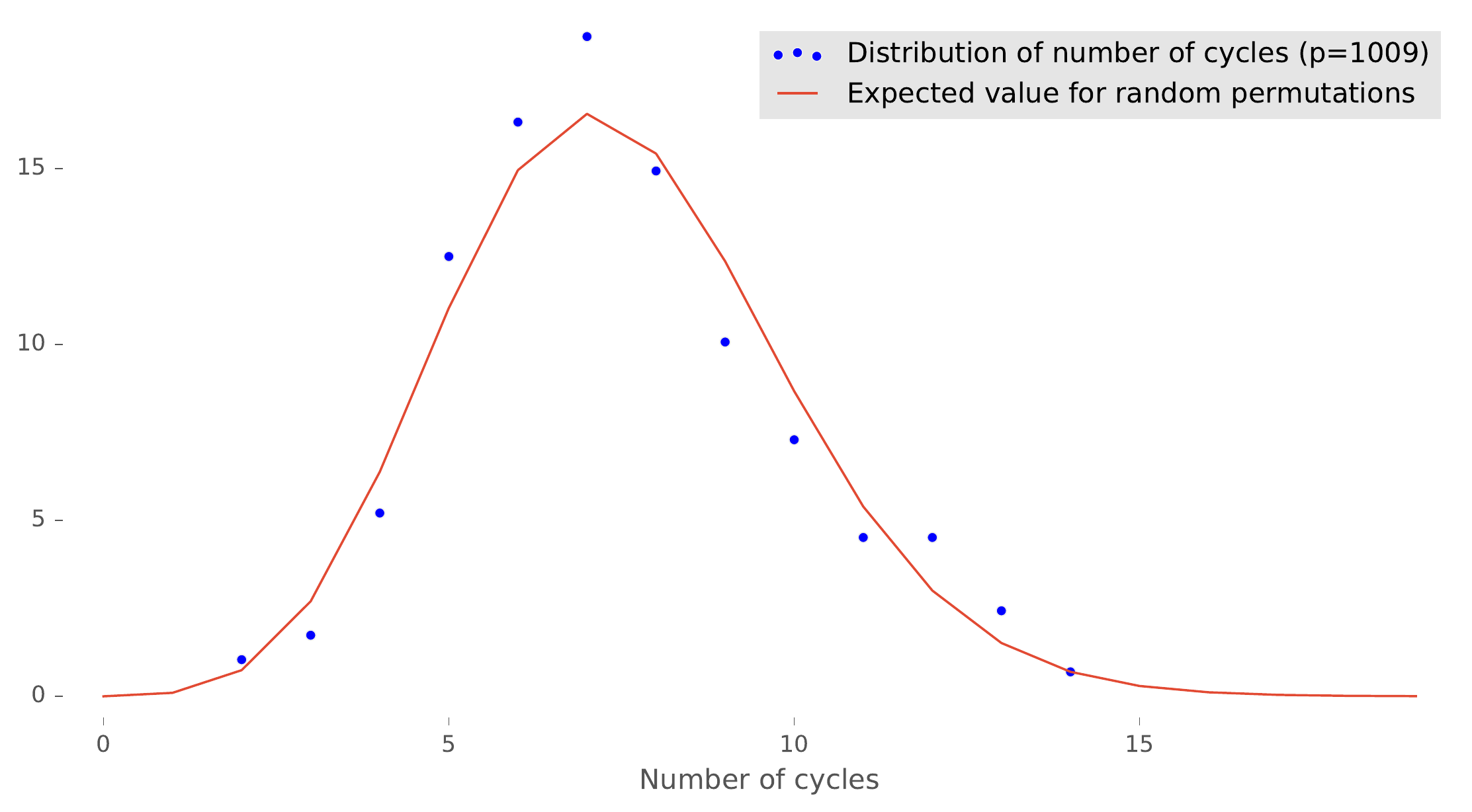}
\caption{Distribution of number of cycles in ElGamal functions on $\mathbb{F}_{1009}$}
\label{fig:elgamalPermCycles}
\end{figure}

\subsection{Number of $k$-cycles in permutations}

Given a random permutation of elements, the number of cycles
of length $k$ is on average $1/k$ \cite{genfun}. In Figure \ref{fig:kcycles}, we 
give the average number of cycles of length $k$ for all $288$ generators of
the multiplicative group in dots.
The experimental results are reasonably close to the theoretical values.

\begin{figure}[ht!]
\centering
\includegraphics[width=\textwidth]{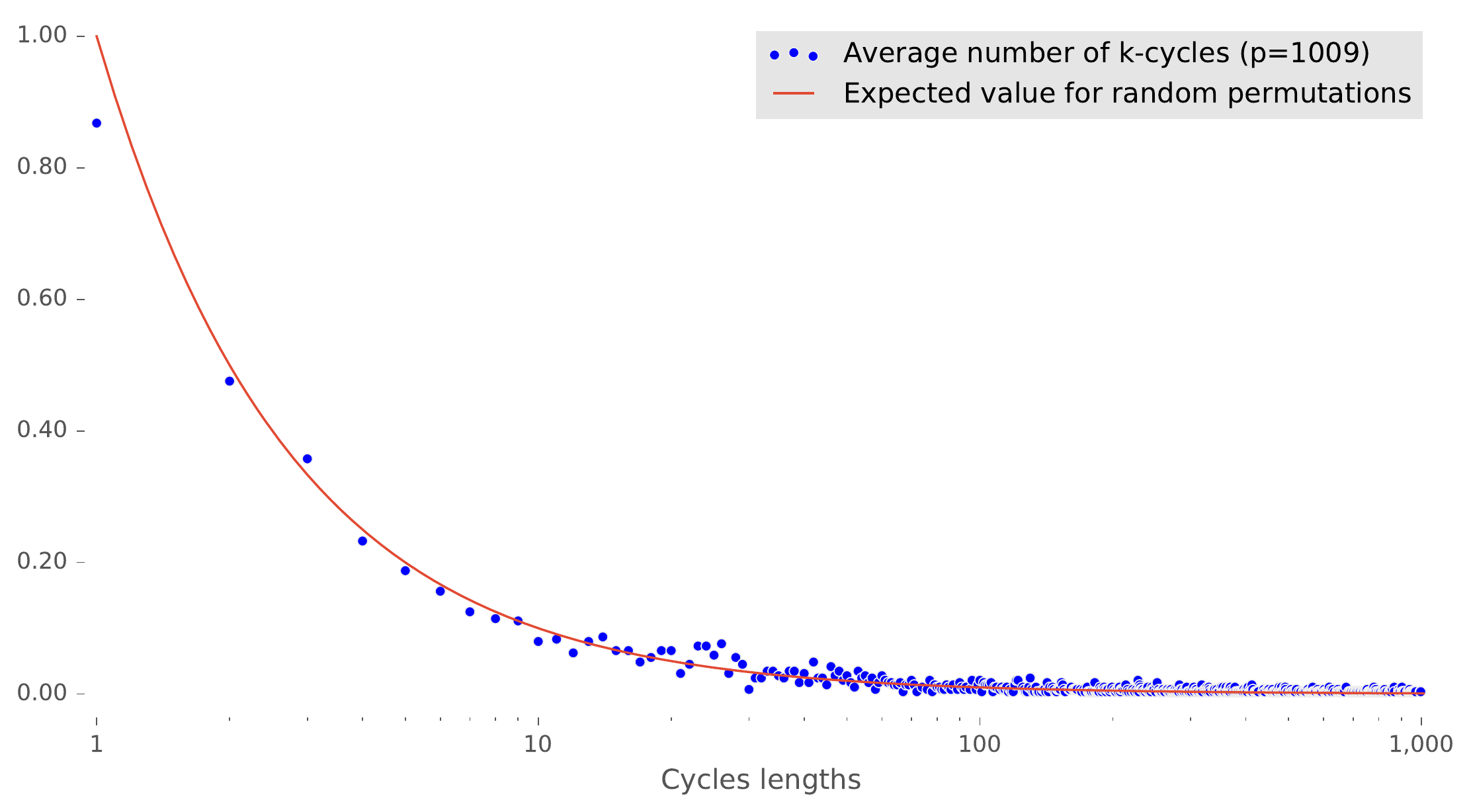}
\caption{Average number of $k$-cycles in in ElGamal functions on $\mathbb{F}_{1009}$}
\label{fig:kcycles}
\end{figure}

For the specific case $k=1$, the average number of fixed points
in random permutations is $1$. The results in Figure \ref{fig:kcycles} are very
close, by a small error margin. Therefore, to better illustrate this property, Figure 
\ref{fig:fixedpoints} shows the average number of fixed points for all generators 
in the multiplicative group for all prime numbers from $2$ to $2111$. As 
expected, the average of fixed points is closely distributed to the theoretical
value. We also note that by increasing $p$, the 
average of fixed points in the experiments gets closer to the expected theoretical
value.

\begin{figure}[ht!]
\centering
\includegraphics[width=\textwidth]{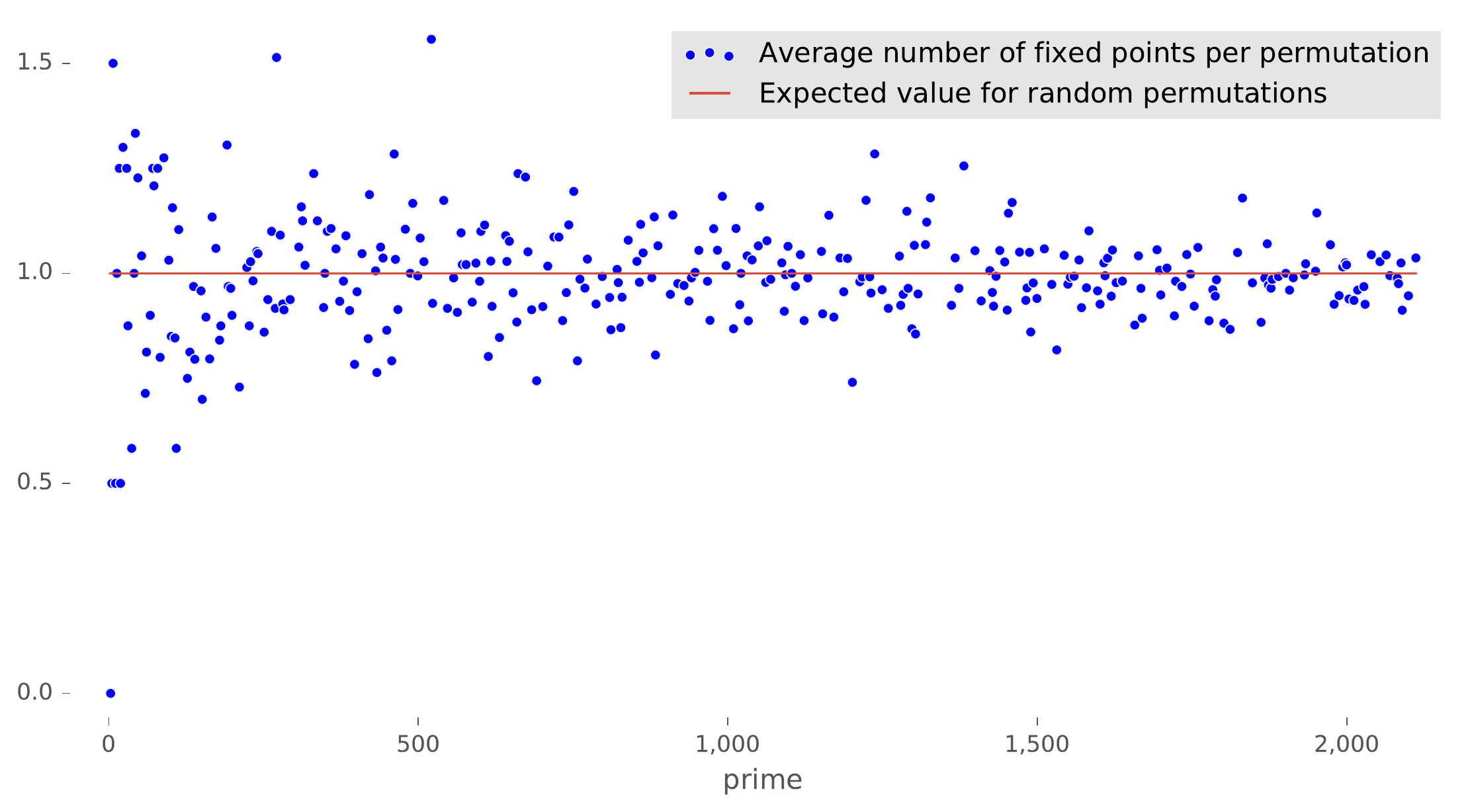}
\caption{Average number of fixed points for all generators of $\mathbb{F}_p$ with $2 \leq p \leq 2111$}
\label{fig:fixedpoints}
\end{figure}

\section{Sidon sets}
A subset $A$ of an abelian group $G$ (written additively) is a \textit{Sidon set} if for every $y\in G\setminus \{0\}$ there exists at most one pair $(a,b)\in A^2$ such that $y=a-b$. Clearly, for any set $A$ there are exactly $\#A$ pairs $(a,b)\in A^2$ for which $0=a-b$,
where  $\# A$ is the cardinality of $A$.

Let $p$ be a prime, $g\in \mathbb Z_p^{\times}$ a generator of the multiplicative group $\mathbb Z_p^{\times}$, 
and identify $\mathbb Z_{p-1}=\{0,1,\dots,p-2\}$ and $\mathbb Z_p=\{0,1,\dots,p-1\}$. 
We consider the additive group $G=\mathbb Z_{p-1}\times \mathbb Z_p$ (using the additive structure of both factors), and the subset 
\begin{equation}\label{eq:DefS}
S=\{(g^x,x): x\in \mathbb Z_{p-1}\}. 
\end{equation}
Thus $S$ is the graph of the discrete logarithm function modulo $p$, since $S=\{(y,\log_g y): y\in \mathbb Z_p\setminus\{0\} \}$,
and, after swapping the coordinates, the graph of the ElGamal function.

The following result is well known; see Cilleruelo \cite{cil12}, Example 2. We include a proof for the sake of completeness.
\begin{lem}\label{thm:Sidon}
The set $S$ in~\eqref{eq:DefS} is a Sidon set in $\mathbb Z_p\times\mathbb Z_{p-1}$.
\end{lem}

\begin{proof} For some $(u,v)\neq (0,0)$ in $\mathbb Z_p\times\mathbb Z_{p-1}$
and $c_1, c_2 \in \Z_{p-1}$, suppose that $(g^{c_1},c_1)-(g^{c_2},c_2)= (u,v)$. Then
\begin{equation}\label{eq:Sidon}
\begin{array}{rcl}
c_1-c_2&\equiv& v \bmod{p-1},\\
g^{c_1}-g^{c_2}&\equiv& u \bmod p.
\end{array} 
\end{equation}
In particular $v \not\equiv 0\bmod {p-1}$, since otherwise $u\equiv g^{c_1}-g^{c_1} \equiv 0 \bmod p$ which contradicts the assumption. 

From the first equation in~\eqref{eq:Sidon} we know that
\[g^{c_1-v} \equiv g^{c_2} \bmod p,\]
hence in the second equation we have
\[g^{c_1}\left(1-g^{-v}\right)\equiv u\bmod p.\]
Since $\left(1-g^{-v}\right)\not\equiv 0 \bmod p$ by the above,
we conclude that
\[g^{c_1}\equiv \left(1-g^{-v}\right)^{-1} u \bmod p,\]
and thus the pair $(c_1,c_2)$ is uniquely determined by $(u,v)$.
\end{proof}
 
\begin{lem}\label{thm:S(A)} Let $\varphi$ be a nontrivial character of $G=\mathbb Z_{p}\times\mathbb Z_{p-1}$ 
and $S$ be the set in~\eqref{eq:DefS}. Then
\[\Big|\sum_{a\in S}\varphi(a)\Big| < (3 (p-1))^{1/2}.\]
\end{lem}

\begin{proof}
Any nontrivial character $\varphi$ of $G$ satisfies
$\sum_{x\in G}\varphi(x)=0$.
Thus, for the set $S-S=\{x\in G:\, x=a-b\text{ for some }a,b\in S\}$ we have
\begin{equation}\label{S-S}
\sum_{x\in S-S}\varphi(x)= - \sum_{x\notin S-S}\varphi(x).
\end{equation}

Since $|z|=(z\cdot \overline{z})^{1/2}$ for a complex number $z$ and $\overline{\varphi(x)}=\varphi(-x)$ for every $x\in G$,
where $\overline z$ denotes the complex conjugate of $z$, it follows that
\begin{align}
\Big|\sum_{a\in S}\varphi(a)\Big|^2& = \Big(\sum_{a\in S}\varphi(a)\Big)\Big(\sum_{b\in S}\varphi(-b)\Big) 
=\sum_{a,b\in S}\varphi(a-b)\nonumber\\
&=\sum_{y\in G}\varphi(y)\cdot \#\{(a,b)\in S^2:\, y=a-b\}\label{eq:S1}.
\end{align}
Since $S$ is a Sidon set by Lemma~\ref{thm:Sidon}, we know that
\[
\#\{(a,b)\in S^2:\, y=a-b\}=\left\{ \begin{array}{cl} 
\#S& \text{if }y=0,\\
1& \text{if }y\in S-S\setminus\{0\},\\
0&\text{otherwise.}
\end{array}\right.
\]
Thus
\begin{align}
\Big|\sum_{a\in S}\varphi(a)\Big|^2& = \#S -1  + \sum_{y\in S-S}\varphi(y) \nonumber \\
& = \#S - 1 - \sum_{y\notin S-S}\varphi(y)\nonumber\\
& \le \#S - 1 +\Big|\sum_{y\notin S-S}\varphi(y)\Big|.\label{eq:S2}
\end{align}
Luckily, we have a complete description of the set $S-S$, since every pair $(a,b)\in S^2$ 
is uniquely determined by the difference $a-b$ unless $a-b=0$, for which we have exactly $\#S = p-1$ options; hence
\begin{equation}\label{eq:Saux}
\#(S-S) = (\#S)^2 - \#S + 1 = (p-1)^2-(p-1)+1 = \#G -2\#S+1
\end{equation}
since $\#G=p(p-1)$. Clearly we have from~\eqref{eq:Saux} that
\begin{equation}\label{eq:S3}
\Big|\sum_{y\notin S-S}\varphi(y)\Big|\le \#G - \#(S-S) = 2\#S -1.
\end{equation}

Combining equations~\eqref{eq:S1},~\eqref{eq:S2} and~\eqref{eq:S3} we have 
\[\Big|\sum_{a\in S}\varphi(a)\Big|^2 \le 3\#S - 2 ,\]
which concludes the proof.
\end{proof}

The following classical result is only included here for the sake of completeness.
\begin{lem}\label{thm:sum}
Let $n$ and $N$ be positive integers with $1\leq N<n$. Then, for any integer $h$
\[\sum_{0\leq a <n}\Big|\sum_{h\leq x <N+h}\exp(2\pi iax/n)\Big|<5n\log n.\]
\end{lem}

\begin{proof}Without loss of generality we will assume that $h=0$, since
\begin{align*}
|\sum_{h\leq x <N+h}\exp(2\pi iax/n)\Big|&=|\sum_{0\leq x <N}\exp(2\pi ia(x+h)/n)\Big|\\
&=\Big|\exp(2\pi iah/n)\sum_{0\leq x <N}\exp(2\pi iax/n)\Big|\\
&=\Big|\sum_{0\leq x <N}\exp(2\pi iax/n)\Big|.
\end{align*}

The contribution of $a=0$ to the sum is precisely $N<n$. 

Observe that for a given $1\le a\le n-1$ the sum 
\[\sum_{0\leq x \leq N}\exp(2\pi iax/n)=1+\exp(2\pi iax/n)+\cdots + \exp(2\pi iax/n)^{N-1}\]
is in fact a geometric progression with ratio $q = \exp(2\pi ia/n) \ne 1$ thus
\[\sum_{0\leq x \leq N}\exp(2\pi iax/n) = \left| \frac{q^N-1}{q-1}\right| \le \frac{2}{|q-1|}.\]
We have
\[
|q-1| = |\exp(2\pi i a/n)-1| = |\exp(\pi i a/n) - \exp(- \pi i a/n)| \\
  = 2 |\sin( \pi  a/n)|.\]
Then
\[|\sin( \pi a/n)| = |\sin( \pi  (a-n)/n)| \ge \frac{2 \min\{a, n-a\}}{n}\]
because $\sin(\alpha) \ge 2\alpha/\pi$ for $0 \le \alpha \le \pi/2$. 
Therefore 
\begin{align}
\sum_{0\leq a < n} \Big|\sum_{0\leq x < N}\exp(2\pi iax/n)\Big| & \le\,  N+  \sum_{0<a < n}\frac{n}{  \min\{a, n-a\}}\nonumber \\
& \le N+2n \sum_{1\le a \leq n/2}\frac{1}{a}.\label{eq:PolyaV}
\end{align}
The proof follows from~\eqref{eq:PolyaV} and the inequality
\[\sum_{1\le a \leq n/2}\frac{1}{a} < 1 +\log(n) ,\]
which holds for any integer $n\ge 2$.
\end{proof}

\begin{thm}\label{thm:principal}
Let $S=\{(g^x,x):x\in \mathbb Z_{p-1}\}$. For any box 
$B=[h+1 \twodots h+N]\times[k+1 \twodots k+M]\subseteq \mathbb Z_p\times\mathbb Z_{p-1}$ we have
\[\left|\#(S\cap B) - \frac{\#B}{p}\right|\le 50 p^{1/2}\log^2p.\] 
Furthermore, if $\#B\in \omega (p^{3/2}\log^2 p)$, then $\#(S\cap B)\sim \#B/p$.
\end{thm}

Here, \[\omega(f) =\{g \colon \mathbb R \to \mathbb R^{+}\, \colon
g(x)/|f(x)| \to 0\text{ if } x\to \infty\}\]
for some $f\colon \mathbb R \to \mathbb R^{+}$.

\begin{proof}
By the orthogonality of characters and separating the contribution of the trivial character $\varphi_0 = 1$, we have
\begin{align*}
\#(S\cap B) &= \frac{1}{p(p-1)}\sum_{\varphi}\sum_{a\in S}\sum_{b\in B}\varphi(a-b)\\
&=\frac{\#B}{p}+\frac{1}{p(p-1)}\sum_{\varphi\neq \varphi_0}\sum_{a\in S}\sum_{b\in B}\varphi(a-b).
\end{align*}
Thus
\begin{align}
\Big|\#(S\cap B) -\frac{\#B}{p}\Big| &= \frac{1}{p(p-1)}\Big|\sum_{\varphi\neq \varphi_0}\sum_{a\in S}\sum_{b\in B}\varphi(a-b)\Big|\nonumber\\
& \le \frac{1}{p(p-1)}\sum_{\varphi\neq \varphi_0}\Big|\sum_{a\in S}\varphi(a)\Big|\Big|\sum_{b\in B}\varphi(b)\Big|\nonumber \\
&\le \frac{1}{p(p-1)}\Big(\max_{\varphi\ne\varphi_0}\Big|\sum_{a\in S}\varphi(a)\Big|\Big) \sum_{\varphi\neq \varphi_0}\Big|\sum_{b\in B}\varphi(b)\Big|.\label{eq:final}
\end{align}
The characters of $G$ act as follows:
\[\varphi((x,y))=\exp\big(2\pi i\big(\tfrac{sx}{p}+\tfrac{ty}{p-1}\big)\big),\qquad \text{for some }(s,t)\in G.\]
Hence we have
\begin{align*}
\sum_{\varphi\neq \varphi_0}\Big|\sum_{b\in B}\varphi(b)\Big|\le & \bigl(\sum_{0 \leq s < p}
|\sum_{h < x \leq h+N}\exp(2\pi i sx/p) |\bigr) \\
& \times
\bigl(\sum_{0 \leq t < p-1}| \!\! \sum_{k < y \leq k+M}\exp(2\pi i ty/(p-1))|\bigr),
\end{align*}
which implies, by Lemma \ref{thm:sum}, that
\begin{equation}\label{eq:Box}
\sum_{\varphi\neq \varphi_0}\Big|\sum_{b\in B}\varphi(b)\Big|< 25p(p-1)\log^2 p.
\end{equation}
By Lemma~\ref{thm:S(A)},
\[\max_{\varphi\ne\varphi_0}\Big|\sum_{a\in S}\varphi(a)\Big|<\sqrt{3(p-1)},\]
which combined with~\eqref{eq:Box} in~\eqref{eq:final} concludes the proof.
\end{proof}

One can show, with a bit more of work, see Cilleruelo \&\ Zumalac\'arregui~\cite{cilzum16}, that in fact 
\[\left|\#(S\cap B) - \frac{\#B}{p}\right|\in O\bigl(p^{1/2}\log_+^2(|B|p^{-3/2})\bigr),\]
which extends slightly the asymptotic range for $\#B$, where our ``big-Oh'' notation,
for a real function $f\colon \mathbb R \to \mathbb R^{+}$, $O(f)$ denotes the following set of functions:
\[O(f)=\left\{g\colon \mathbb R \to \mathbb R |\, \exists\ C>0\ \text{ with } |g(x)|\le Cf(x) \text{ for sufficiently large }x\right\},\]
and $\log_+(x) = \max \{\ln (x), 1\}$ for $x \in \mathbb R^+$.
The implied asymptotics are for growing $p$.
In fact, in~\cite{cilzum16} this result was obtained for a much larger family of dense Sidon sets. 

Igor Shparlinski has pointed out to us that one can obtain similar asymptotic results
with the exponential sum machinery. However, that method is unlikely to yield explicit estimates,
without ``O''-term.

\section{Conclusion}
We have shown, both experimentally and theoretically, some randomness properties
of the ElGamal function over $\Z_p$ for a prime $p$. Many questions along these lines remain open:
\begin{itemize}
\item
stronger results, perhaps even pseudorandomness,
\item
other groups for $G$, for example, elliptic curves,
\item
similar questions about the Schnorr function, where $G$ is a ``small'' subgroup of
a ``large'' group $ \mathbb Z_p$.
\end{itemize}

\section{Acknowledgements}
Part of this work was done by the first three authors at LabSEC of the Universidade
Federal de Santa Catarina, Brazil. We thank Daniel Panario for providing
background material, and Ricardo Cust\'odio for making our collaboration
possible.
The collaboration with Ana Zumalac\'arregui was begun at a meeting
at the CIRM, Luminy, France, whose support we gratefully acknowledge.
The second author was supported by the B-IT Foundation and the Land Nordrhein-Westfalen. The last author was supported by Australian Research Council Grants DP140100118.

\bibliographystyle{plain}
\bibliography{references}

\begin{thebibliography}{10}

\bibitem{canettiDDH}
Ran Canetti, John Friedlander, Sergei Konyagin, Michael Larsen, Daniel Lieman,
  and Igor Shparlinski.
\newblock On the statistical properties of {Diffie}-{Hellman} distributions.
\newblock {\em Israel Journal of Mathematics}, 120:23--46, 2000.

\bibitem{cil12}
J.~Cilleruelo.
\newblock Combinatorial problems in finite fields and {Sidon} sets.
\newblock {\em Combinatorica}, 32(5):497--511, 2012.

\bibitem{cilzum16}
Javier Cilleruelo and Ana Zumalac\'arregui.
\newblock Saving the logarithmic factor in the error term estimates of some
  congruence problems.
\newblock {\em Math. Z.}, 286(1-2):545--558, 2017.

\bibitem{CobVajZah}
Cristian Cobeli, Marian V\^aj\^aitu, and Alexandru Zaharescu.
\newblock On the set {$ax+bg^x\pmod p$}.
\newblock {\em Port. Math. (N.S.)}, 59(2):195--203, 2002.

\bibitem{elgamal}
Taher ElGamal.
\newblock A public key cryptosystem and a signature scheme based on discrete
  logarithms.
\newblock {\em IEEE Transactions on Information Theory}, IT-31(4):469--472,
  1985.

\bibitem{felkur16}
Adam~Tyler Felix and P{\"a}r Kurlberg.
\newblock On the fixed points of the map $x \mapsto x^x$ modulo a prime, {II}.
\newblock 2016. {\tt arXiv:1607.04948}.

\bibitem{cryptoschool}
Joachim von~zur Gathen.
\newblock {\em CryptoSchool}.
\newblock Springer, 2015.

\bibitem{konluc16}
Sergei~V. Konyagin, Florian Luca, Bernard Mans, Luke Mathieson, Min Sha, and
  Igor~E. Shparlinski.
\newblock Functional graphs of polynomials over finite fields.
\newblock {\em J. Combin. Theory Ser. B}, 116:87--122, 2016.

\bibitem{kurluc15}
P{\"a}r Kurlberg, Florian Luca, and Igor~E.\ Shparlinski.
\newblock On the fixed points of the map $x \mapsto x^x$ modulo a prime.
\newblock {\em Mathematical Research Letters}, 22(01):141--168, 2015.

\bibitem{mansha17}
Bernard Mans, Min Sha, Igor~E.\ Shparlinski, and Daniel Sutantyo.
\newblock On functional graphs of quadratic polynomials, 2017. {\tt
  arXiv:1706.04734}.

\bibitem{marpan16}
Rodrigo S.~V. Martins and Daniel Panario.
\newblock On the heuristic of approximating polynomials over finite fields by
  random mappings.
\newblock {\em International Journal of Number Theory}, 12:1987--2016, 2016.
  Erratum pages 2041-2042.

\bibitem{genfun}
H.~Wilf.
\newblock {\em Generatingfunctionology}.
\newblock Academic Press, New York, 1990.

\end{thebibliography}

\vspace{2ex}

Author addresses:

 Lucas Boppr\'e Niehues and Lucas Pandolfo Perin, LabSEC, Universidade
Federal de Santa Catarina, Brazil.\\ {\tt lucasboppre@gmail.com} and {\tt lucas.perin@posgrad.ufsc.br}

\smallskip

 Joachim von zur Gathen, B-IT, Universit\"at Bonn, Germany.\\ {\tt gathen@bit.uni-bonn.de}
 
 \smallskip
 
 Ana Zumalac\'arregui, University of New South Wales, Sydney, Australia.\\ {\tt a.zumalacarregui@unsw.edu.au}
\end{document}